\documentclass{amsart}
\usepackage{amssymb,latexsym, amsmath, amscd, array, graphicx}

\swapnumbers
\numberwithin{equation}{section}

\theoremstyle{plain}

\newtheorem{thm}{Theorem}[section]

\newtheorem{theorem}[thm]{Theorem}

\newtheorem{question}[thm]{Problem}
\newtheorem{prop}[thm]{Proposition}

\theoremstyle{definition}

\newtheorem{definition}[thm]{Definition}

\newtheorem{remark[thm]}{Remark}

\def\Ord{\protect\operatorname{Ord}}

\def\cat{\protect\operatorname{cat}}

\def\cd{\protect\operatorname{cd}}

\def\scr{\mathcal}

\def\1{\hbox{\rm\rlap {1}\hskip.03in{\rom I}}}
\def\Bbbone{{\rm1\mathchoice{\kern-0.25em}{\kern-0.25em}
{\kern-0.2em}{\kern-0.2em}I}}

\long\def\forget#1\forgotten{} %

\begin{document}
\title[On topological complexity of twisted products ]
{On topological complexity of twisted products}

\author[A.~Dranishnikov]
{Alexander Dranishnikov}
\address{A. Dranishnikov, Department of Mathematics, University
of Florida, 358 Little Hall, Gainesville, FL 32611-8105, USA}
\email{dranish@math.ufl.edu}
\thanks{Supported by NSF, grant DMS-1304627}
\subjclass[2000]
{Primary 55M30; 
Secondary 57N65, 54F45}  

\begin{abstract} We provide an upper bound on the topological complexity of twisted products.
We use it to give an estimate $$TC(X)\le TC(\pi_1(X))+\dim X$$
of the topological complexity of a space in terms of its dimension and the complexity of its fundamental group. 
\end{abstract}

\maketitle

\section{Introduction}

 The {\em topological complexity}  $TC(X)$ of a space $X$ was defined by Farber~\cite{F1} as the minimum integer $k$ 
such that $X\times X$ admits an open cover $U_0,\dots, U_k$ by $k+1$ sets such that for each $U_i$ there is a continuous  motion planning algorithm $s_i$, i.e., a continuous map $s_i:U_i\to PX$ to the path space  $PX=X^{[0,1]}$ such that
$s_i(x,y)(0)=x$ and $s_i(x,y)(1)=y$ for all $(x,y)\in U_i$. We note that here we defined the reduced topological complexity. In the original definition 
the enumeration of $U_i$ goes from 1 to $k$. Thus, it defines the nonreduced topological complexity which  is by one larger.

The topological complexity is homotopy invariant. Therefore one can define the topological complexity $TC(\pi)$
of a discrete group $\pi$ as $TC(K(\pi,1))$ where $K(G,1)$ is an Eilenberg-Maclane complex. It is known that $$cd(\pi)\le TC(\pi)\le 2cd(\pi)$$ 
where $cd(\pi)$ is the cohomological dimension of a group $\pi$~\cite{Br}.
M. Farber proposed a natural question:
\begin{question}[Farber]
What kind of a discrete group invariant is $TC(\pi)$?
\end{question}
Note that for groups with infinite cohomological dimension $TC(\pi)=\infty$. Thus, Farber's question makes sense
only for groups admitting a finite dimensional Eilenberg-MacLane complex, in particular, for torsion free groups.
It is known that the reduced Lusternik-Schnirelman category of a group $\pi$ agrees with the cohomological dimension, $$\cat(K(\pi,1))=cd(\pi).$$
In view of the equality $TC(G)=\cat G$~\cite{F3} for  
all topological groups $G$ and the fact that $K(\pi,1)$ is homotopy equivalent to a topological group for abelian 
$\pi$, we obtain that $TC(\pi)=cd(\pi)$ for abelian groups.
For free nonabelian groups the other bound is taken, $TC(\pi)=2cd(\pi)$.
Using the wedge formula for the topological complexity~\cite{Dr3} Yu. Rudyak noticed that
for all values $k$ and $n$ with $n\le k\le 2n$ there is a group $\pi$ with $cd(\pi)=n$ and $TC(\pi)=k$~\cite{Ru}.

In this paper we investigate how the topological complexity of the fundamental group could help
to estimate the topological complexity of a space.
A. Costa and M. Farber~\cite{CF} obtained the following upper bound
$$
TC(X)\le 2\cd(\pi)+\dim X
$$
for
the topological complexity of a space $X$ with the fundamental group $\pi$.
Their result is parallel to the estimate for the Lusternik-Schnirelmann category~\cite{Dr1}(see also~\cite{Dr2}):
$$\cat X\le cd(\pi)+\frac{1}{2}\dim X.$$

In this paper we improve the Costa-Farber inequality to the following
$$
TC(X)\le TC(\pi)+\dim X.
$$
This result was obtained as a corollary of an  upper bound formula for the topological complexity
of a twisted product.
It is known that in the case of the Cartesian product $X=B\times F$ there is a formula
$$
TC(X)\le TC(B)+TC(F).
$$
In the case of a twisted product $X=B\tilde{\times}F$ over $B$ with the fiber $F$ and the structure group $G$ we prove the formula
$$
TC(X)\le TC(B)+TC^*_G(F)
$$
where $TC^*_G(F)$ is a version of the equivariant topological complexity introduced in this paper.
We note that our version of the equivariant topological complexity differs from those defined by Colman-Grant~\cite{CG} and Lubawski-Marzantowicz~\cite{LM}.
We call it the {\em strongly equivariant topological complexity}.

\section{Preliminaries}
Inspired by the work of Kolmogorov~\cite{K} on Hilbert's 13th problem, Ostrand~\cite{Os}
gave a characterization of  dimension in terms of $k$-covers. In this paper we apply his technique
to the topological complexity.

A family $\scr U$ of subsets of $X$ is called a {\em $k$-cover},
$k\in N$ if every its subfamily of $k$ elements forms a cover of $X$. 
The order $\Ord_x\scr U$ of a cover $\scr U$ at a point $x\in X$ is the number of elements in $\scr U$
that contain $x$.
The
following is obvious.
\begin{prop}\label{n-cover}
A family $\scr U$ that consists of $m$ subsets of $X$ is an
$(n+1)$-cover of $X$ if and only if $\Ord_x\scr U\ge m-n$ for all
$x\in X$.
\end{prop}
\begin{proof}
If $\Ord_x\scr U< m-n$ for some $x\in X$, then $n+1=m-(m-n)+1$
elements of $\scr U$ do not cover $x$.

If $n+1$ elements of $\scr U$ do not cover some $x$, then
$\Ord_x\scr U\le m-(n+1)<m-n$.
\end{proof}

\begin{definition} Given $\Delta\subset Z$,
a subset $U\subset Z$ is called {\em deformable to $\Delta$} if there is a homotopy
$h_t:U\to Z$ with $h_0: U\to Z$ the inclusion and $h_1(U)\subset\Delta$.
An open cover $\scr U=\{U_0, U_1,\dots, U_n\}$ of $Z$ is called {\em
$\Delta$-deformable} if  each $U_i$ is deformable to  $\Delta$.
If a group $G$ acts on $Z$ and $\Delta$ is an invariant set, we call a subset $U$ equivariantly deformable to $\Delta$
if the above homotopy $h_t$ is an equivariant map for each $t\in[0,1]$.
\end{definition}
The following is well-known:
\begin{prop}
The topological complexity $TC(X)$ of a space $X$ is the minimum number $k$ 
such that $X\times X$ admits an open cover $U_0,\dots, U_k$ by $\Delta(X)$-deformable sets where
$\Delta(X)=\{(x,x)\in X\times X\mid x\in X\}$ is the diagonal in $X\times X$.
\end{prop}
We call an action of a topological group $G$ on a locally compact metric space $X$ {\em proper} if for every compact
$C\subset X$ the set $\{g\in G\mid g(C)\cap C\ne\emptyset\}\subset G$ is compact.
We recall that the orbit space $X/G$ of such action is always completely regular (Proposition 1.28~\cite{Pa}).
\begin{theorem}\label{criterion}
Let $\{U_0',\dots,U_n'\}$ be an open cover by $G$-invariant sets of a locally compact metric space
$Z$ with a proper action of a topological group $G$ on it. Then  for any $m=n,n+1,\dots,\infty$ there is an  open 
$(n+1)$-cover of $Z$ by $G$-invariant sets $\{U_k\}_{k=0}^{m}$ such that $U_k=U_k'$ for
$k\le n$ and $U_k=\cup_{i=0}^nV_i$ is a disjoint union with
$V_i\subset U_i$ for $k>n$.

In particular, if $\{U_0',\dots,U_n'\}$ is $\Delta$-deformable for some $\Delta\subset Z$, then the
cover $\{U_k\}_{k=0}^{m}$ is $\Delta$-deformable. 

Moreover, if $\{U_0',\dots,U_n'\}$ is equivariantly $\Delta$-deformable for some subgroup $H\subset G$, then the
cover $\{U_k\}_{k=0}^{m}$ is equivariantly $\Delta$-deformable. 
\end{theorem}
\begin{proof}
We construct the family $\{U_i\}_{i=0}^{m}$ by induction on $m$. For
$m=n$ we take $U_i=U_i'$.

Let $\scr U=\{ U_0,\dots, U_{m-1}\}$ be the corresponding family for
$m> n$. By Proposition~\ref{n-cover}, $\Ord_z\scr U\ge m-n$. Consider the set
$Y=\{y\in Z\mid \Ord_y\scr U= m-n\}$. Clearly, it is a closed $G$-invariant subset
of $Z$. If $Y=\emptyset$, then by Proposition~\ref{n-cover} $\ \scr U$
is an $n$-cover and we can add $U_m=U_0$ to obtain a desired
$(n+1)$-cover. Assume that $Y\ne\emptyset$. We show that for every
$i\le n$ the set $Y\cap U_i$ is closed in $Z$. Let $x$ be a limit
point of $Y\cap U_i$ that does not belong to $U_i$. Let
$U_{i_1},\dots,U_{i_{m-n}}$ be the sets of the cover $\scr U$ that
contain $x\in Y$. The limit point condition implies that
$$(U_{i_1}\cap\dots\cap U_{i_{m-n}})\cap (Y\cap U_i)\ne\emptyset.$$
Then $\Ord_y\scr U=m-n+1$ for all $y\in Y\cap U_i\cap
U_{i_0}\cap\dots\cap U_{i_{m-n}}$. Contradiction.

We define recursively $F_0=Y\cap U_0$ and $$F_{i+1}=Y\cap
U_{i+1}\setminus (\bigcup_{k=0}^iU_k).$$ 
Since the sets $U_0,\dots U_n$ cover $Z$, this process stops at $i=n$.
Note that $\{F_i\}_{i=0}^n$ is
a disjoint finite family of closed $G$-invariant subsets with $\cup_{i=0}^nF_i=Y$.
Let $q:Z\to Z/G$ be the projection to the orbit space. Since the action is proper, by the combination of the Urysohn Metrization theorem
and Palais result mentioned  before this theorem the orbit space $Z/G$ is metrizable.
Note that the sets $q(F_i)$ are disjoint. 
Since $q$ is a quotient map and $F_i$ and $U_i$ are $G$-invariant, the sets $q(F_i)$  and   $q(U_i)$ are closed and open respectively. 
By taking disjoint open neighborhoods
of $q(F_i)$  lying in $q(U_i)$ we
fix open $G$-invariant disjoint neighborhoods $V_i$ of $F_i$ with $V_i\subset
U_i$. 
We define $U_{m}=\cup_{i=0}^nV_i$. In view of
Proposition~\ref{n-cover}, $U_0,\dots, U_{m-1},U_{m}$ is an
$(n+1)$-cover.

Clearly, the deformations of $U_i$  to $\Delta$, $i\le n$, define a deformation of $U_{m}$
to $\Delta$. If the deformations of $U_i$, $i\le n$, to $\Delta$ are equivariant for some
subgroup $ H\subset G$, then the deformation of $U_m$ is equivariant.
\end{proof}

The following is well known:
\begin{prop}\label{nerve}
Let $X$ be a metric space, $A$ a subset of $X$ and
$\mathcal V' = \{V'_i\}_{i\in J}$ a cover of $A$ by sets open in $A$. Then
$\mathcal V'$ can be extended to a cover $\mathcal V = \{V_i\}_{i\in J}$ of
$A$ by sets open in $X$ with the same nerve and such that $V_i\cap A=V_i'$ for all $i\in J$.
\end{prop}
\begin{proof} The required extension $\mathcal V = \{V_i\}$ can be defined by the formula $$V_i = \bigcup_{a \in V'_i}B(a, \frac{d(a, A - V'_i)}{2}),\ \  i\in J$$
(see Proposition 3.1~\cite{Sr}).
\end{proof}

\section{Strongly equivariant topological complexity}

Let $G$ act on $Y$, we define the {\em strongly equivariant topological complexity} $TC_G^*(Y)$ to be the minimal 
integer $k$ such that there
is an open cover of $Y\times Y$ by $(G\times G)$-invariant sets $U_0,\dots, U_k$ such that for each $i$ there is
a $G$-equivariant map $\phi_i:U_i\to Y^I$ for the diagonal action on $U_i$ such that $\phi(y_0,y_1)(0)=y_0$ and
$\phi(y_0,y_1)(1)=y_1$. 

Note that in the definition of the equivariant topological complexity $TC_G(Y)$  Colman and Grant~\cite{CG} 
ask the sets $U_i$ to be $G$-invariant. It follows from the definitions that $TC_G(Y)\le TC^*_G(Y)$.
Another version of an equivariant topological complexity 
called the symmetric topological complexity $STC_G(Y)$ was defined by Lubawski and Marzantowicz~\cite{LM}. 
They require the existence of $(G\times G)$-equivariant deformations of $U_i$ to the $(G\times G)$-saturation of the diagonal.
We note that for free actions on simply connected spaces, 
$TC_G^*(Y)\le STC_G(Y)$ .

Let a group $G$ act on $F$ and let $p:EG\to BG$ be the projection to the orbit space of the universal
$G$-space. Then the universal $F$-bundle $p_F:F\times_GEG\to BG$ is obtained by taking the orbit space
$F\times_GEG$ of the diagonal action on $F\times EG$. We say that an $F$-bundle $f:X\to Y$ has $G$ as the structure group
if there is a map $g:Y\to BG$ such that $f$ is the pull-back $g^*(p_F)$ of the universal $F$-bundle.
In that case we call $X$ the twisted product over $Y$ with the fiber $F$ and the structure group $G$.

\begin{thm}\label{2}
Suppose that $p:X\to B$ is a $F$-bundle between locally compact metric ANR-spaces with the structure group $G$ acting properly on $F$. 
Then $$TC(X)\le TC(B)+TC_{G}^*(F).$$
\end{thm}
\begin{proof}
Let $TC(B)=n$ and $TC_{G}^*(F)=m$.
By Theorem~\ref{criterion} applied for the trivial group there is  an $(n+1)$-cover $U_0,\dots, U_{m+n}$ 
of $B\times B$ by $\Delta(B)$-deformable open sets.
By Theorem~\ref{criterion} applied for the $(G\times G)$-action on $F\times F$ with the diagonal
subgroup $G=\Delta(G)\subset G\times G$ there is an $(m+1)$-cover  $V_0,\dots, V_{m+n}$ of $F\times F$ 
by $(G\times G)$-invariant $G$-equivariantly deformable to 
$\Delta(F)$ open sets. 

The bundle $p\times p:X\times X\to B\times B$ is the pull back in the diagram
$$
\begin{CD}
X\times X @>g'>> (F\times F)\times_{(G\times G)}E(G\times G)\\
@Vp\times pVV @VqVV\\
B\times B @>g>> B(G\times G)\\
\end{CD}
$$
Note that the family
$$
O_k=V_k\times_{(G\times G)}E(G\times G), \ \  k=0,\dots, m+n.
$$
is  an $(m+1)$-cover of the space $(F\times F)\times_{(G\times G)}E(G\times G)$. 

Let $O'_k=(g')^{-1}(O_k)$. We define 
$$W_k=(p\times p)^{-1}(U_k)\cap O_k',\ \ k=0,\dots,m+n.$$

We claim that the sets $W_k$, $k=0,\dots n+m$, cover  the space $X\times X$.

Let $(x,y)\in X\times X$. By Proposition~\ref{n-cover} the point $(p(x),p(y))\in B\times B$ is covered at least by $m+1$ elements from $\{U_i\}$. 
Let $(p(x),p(y))\in U_{i_0}\cap\dots\cap U_{i_m}$. By the assumption, the family $V_{i_0},\dots,V_{i_m}$ covers $F\times F$. Therefore, the family $O_{i_0},\dots,O_{i_m}$ covers $F\times F\times_{(G\times G)}E(G\times G)$ and
hence the family $$(g')^{-1}(O_{i_0}),\dots,(g')^{-1}(O_{i_m})$$ covers $X\times X$.
Thus, $(x,y)\in (g')^{-1}(O_{i_s})$ for some $s$.
Then $(x,y)\in W_{i_s}$.

Next we show that each set $W_k$ is deformable to $\Delta(X)$. The homotopy lifting property of the fiber bundle $f:O'_k\to B\times B$
where $f=(p\times p)|_{O'_k}$ and the fact that the inclusion $j:U_k\to B\times B$ is deformable to the diagonal  $\Delta(B)$ imply that
the set $W_k$ can be deformed in $O_k'\subset X\times X$ to the preimage $f^{-1}(\Delta(B))$ of the diagonal. 

We note that the bundle $q$ over the diagonal $\Delta(BG)\subset\Delta(BG)\times\Delta(BG)$ is isomorphic to the twisted product
bundle $q_0:(F\times F)\times_GEG\to BG$ for the diagonal action of $G$ on $F\times F$. Let $\phi^k_t:V_k\to F$ be a $G$-equivariant homotopy
with $\phi^k_0(x_1,x_2)=x_1$ and  $\phi^k_1(x_1,x_2)=x_2$. We define a deformation $\Phi^k_t:V_k\to F\times F$ by the formula
$\Phi^k_t(x_1, x_2)=\phi^k_t(x_1,x_2)\times x_2$. Note that it is $G$-equivariant for the diagonal action, $\Phi_0^k=id$, and $\Phi_1^k(V_k)\subset\Delta(F)$. The deformation $\Phi_t^k$ defines a fiberwise deformation of $V_k\times_GEG$ 
in $(F\times F)\times_GEG$ to the space $\Delta(F)\times_G EG$. Thus, it defines a fiberwise deformation of $O_k$ over the diagonal $\Delta(BG)$
to $$\Delta(F\times_GEG)\subset (F\times_GEG)\times(F\times_GEG)=(F\times F)\times_{(G\times G)}E(G\times G).$$
In the pull-back diagram this defines a fiberwise deformation to $\Delta(X)$ of the set $O_k'$ over $\Delta(B)$ which is $f^{-1}(\Delta(B))$.

Thus, the concatenation of the above two deformations define a continuous deformation of $W_k$
to the diagonal $\Delta(X)$.
\end{proof}
\begin{prop}\label{dim}
Suppose that a discrete group $\pi$ acts freely and properly on a simply connected locally compact ANR space $Y$. Then $TC^*_{\pi}(Y)\le \dim Y$. 
\end{prop}
\begin{proof} Let $X= Y/\pi$ and let $\dim Y=k$. Since $\dim X=k$, $\dim(X\times X)\le 2k$. It follows from the classical 
dimension theory that there are 1-dimensional sets $S_0,\dots, S_k$ that cover $X\times X$~\cite{En}.
In view of Proposition~\ref{nerve} and 1-dimensionality of $S_i$ there is an arbitrary small cover $\mathcal U_i$ of $S_i$ of order $\le 2$
by open in $X\times X$ sets with compact closure.  Let $W_i=\cup_{U\in\mathcal U_i}U$. Since $X\times X$ 
is an ANR, we may assume that for each $i$ there is a projection to the nerve $\phi_i:W_i\to K_i=N(\mathcal U_i)$ 
 and a map $\xi_i:K_i\to  X\times X$ such that the composition $\xi_i\circ\phi_i:W_i\to  X\times X$ is homotopic 
to the inclusion $W_i\subset  X\times X$.

Let $\bar q:Y\times Y\to X\times X$ be the projection onto the orbit space of the  action of $\pi\times\pi$ and 
let $q:Y\times Y\to Y\times_{\pi}Y$ be projection onto the orbit space of the diagonal subgroup action. Then there is a connecting projection
$p:Y\times_{\pi}Y\to X\times X$, $\bar q=p\circ q$. Since the actions are free,
$p$ is a covering map. We may assume that $\mathcal U_0\cup\dots\cup\mathcal U_k$ is a cover of $X\times X$ by even for $p$ sets. Thus, for $U\in\mathcal U_i$, $$p^{-1}(U)=\coprod_{\gamma\in J}\tilde U_{\gamma}$$ and
the restriction of $p|_{\tilde U_{\gamma}}:\tilde U_{\gamma}\to U$ is a homeomorphism for all $\gamma\in J$. 
Moreover, we may assume that there is an even cover  $\mathcal V$ of $X\times X$ and a homotopy $h^i_t$ between the inclusion $W_i\to X\times X$ and $\xi_i\circ\phi_i$ such that for each $i$ and $U\in\mathcal U_i$ there is $V\in\mathcal V$ such that the image of $U\times[0,1]$ under that homotopy is contained in $V$.

Then $p$ induces a simplicial covering map $p_i':K_i'\to K_i$ of the nerve $K'_i$ of the cover
$\mathcal U_i'=\{\tilde U_{\gamma}\mid U\in\mathcal U_i, \gamma\in J\}$ onto the nerve $K_i$. 
In the above notations $p'_i$ takes a vertex corresponding to $\tilde U_{\gamma}$ to the vertex defined by $U$.

Let $W_i'=p^{-1}(W_i)$. Then we claim that there are  maps $$\phi_i':W_i'\to K_i' \ \ \text{and}\  \ \xi_i':K_i'\to Y\times_{\pi}Y$$ that cover $\phi_i$ and $\xi_i$ with $\xi_i'\circ\phi_i':W_i'\to  Y\times_{\pi}Y$
homotopic to the inclusion $W_i'\subset  Y\times_{\pi}Y$. Indeed, the projection to the nerve $\phi_i:W_i\to K_i$ is defined by means of  a partition
of unity $\{f_j\}$ subordinated to the cover $\mathcal U_i=\{U^i_j\}$ of $W_i$ as 
$\phi_i(x)=(f_j(x))\in\ell_2(K_i^{(0)})$. Here we realize $K_i$ in the standard simplex in the Hilbert space spanned by the vertices $K_i^{(0)}$.
Let $f_{j,\gamma}$ be the extension to $W_i'$ of the composition $f_j\circ p|_{\tilde U_{\gamma}}:\tilde U_{\gamma}\to[0,1]$ by  0. Then $\{f_{j,\gamma}\}$ is the partition of unity
on $W_i'$ subordinated to $\mathcal U'$ that defines a projection $\phi_i'$ and the corresponding diagram is the pull-back diagram.

The homotopy $g^i_t=h^i_t\circ  p:W_i'\to X\times X$ admits a lift $\bar g^i_t:W_i'\to Y\times_{\pi}Y$ with $g^i_0:W_i'\to Y\times_{\pi}Y$ the inclusion.
Since for every $x'\in K_i'$ the map $\bar g^i_1$ coincides with $(p|_{V_{\alpha}})^{-1}\circ h^i_1\circ p$ on $\phi_i^{-1}(x')$ where 
$p|_{V_{\alpha}}:V_{\alpha}:\to V$ is a homeomorphism for some $V\in\mathcal V$ and $\alpha$, the set $$\bar g^i_1((\phi_i')^{-1}(x'))=
(p|_{V_{\alpha}})^{-1}(\phi_i(p'_i(x')))$$ consists of one point. Hence
 $\bar g^i_1$ factors through $\phi_i'$.

Note that $X=\Delta(Y)/\pi$ is naturally embedded in $Y\times_{\pi}Y$. We show that the sets $W_i'=p^{-1}(W_i)$ are deformable to $X$.
The construction of a deformation of $\xi_i'$ to a map to $X$ goes as follows. First we do it on all vertices 
$v\in (K_i')^{(0)}$ by fixing a path $p_v:[0,1]\to  Y\times_{\pi}Y$ with $p_v(0)=\xi_i'(v)$ and $p_v(1)\in X\subset Y\times_{\pi}Y$.
Note that the inclusion $X\to  Y\times_{\pi}Y$ induces an isomorphism of the fundamental groups. Then the homotopy exact 
sequence of  pair implies that
$\pi_1( Y\times_{\pi}Y,X)=0$. Therefore, for every edge $[u,v]\subset K_i'$ the product of the paths   $\bar p_u\ast\xi_i'|_{[u,v]}\ast p_v$, 
is path homotopic to a path in $X$. Here $\bar p_v$ denotes the inverse path for $p_v$. 
This defines a deformation of $\xi_i'|_{[u,v]}$ to $X$ that 
agrees with $p_u$ and $p_v$. All such deformations of edges  define a deformation of $\xi_i'$ to $X$. 
This deformation together with a homotopy of the inclusion $W_i'\subset  Y\times_{\pi}Y$ to $\xi_i'\circ\phi_i'$ 
defines a deformation of $W_i'$ to $X$.

Let $V_i=q^{-1}(W_i')$. Then each $V_i$ is $(\pi\times\pi)$-invariant since $V_i=\bar q^{-1}(W_i)$. A deformation of $W_i'$ to $\Delta(Y)/\pi$
defines a $G$-equivariant deformation of $V_i$ to $\Delta(Y)$
Thus, an open cover $V_i=q^{-1}(W'_i)$, $i=0,\dots, k$, of $Y\times Y$ satisfies all conditions from the definition of $TC^*_{\pi}(Y)$.
\end{proof}

\begin{thm}
For a CW complex $X$ with the fundamental group $\pi$  there is the inequality
$$
TC(X)\le TC(\pi)+\dim X.
$$
\end{thm}

\begin{proof}
We apply Theorem~\ref{2} to the bundle $p:\widetilde X\times_{\pi}E\pi\to B\pi$ with the structure group $\pi$ and the fiber $\widetilde X$, the universal cover of $X$.
Note that since $E\pi$ is contractible, the map $\widetilde X\times_{\pi}E\pi\to X$
induced by the projection $pr_1: \widetilde X\times E\pi\to\widetilde X$ to the first factor is a homotopy equivalence.
We apply Proposition~\ref{dim} to complete the proof.
\end{proof}


\begin{thebibliography}{[CLOT]}
\bibliographystyle{amsalpha}



\bibitem[CF]{CF}  Costa, Armindo; Farber, Michael, Motion planning in spaces with small fundamental groups. Commun. Contemp. Math. 12 (2010), no. 1, 107-119.

\bibitem[CG]{CG}  Colman, Hellen; Grant, Mark Equivariant topological complexity. Algebr. Geom. Topol. 12 (2012), no. 4, 2299-2316.

\bibitem[Br]{Br} K. Brown, {\em Cohomology of groups}, Springer 1982.

\bibitem[Dr1]{Dr1} A. Dranishnikov,  On the Lusternik-Schnirelmann category of spaces with 2-dimensional fundamental group. Proc. Amer. Math. Soc. 137 (2009), no. 4, 1489-1497.

\bibitem[Dr2]{Dr2} A. Dranishnikov, The Lusternik-Schnirelmann category and the fundamental group. Algebr. Geom. Topol. 10 (2010), no. 2, 917–924. 

\bibitem[Dr3]{Dr3} A. Dranishnikov, Topological complexity of wedges and covering maps, Proc. AMS, to appear.

\bibitem[En]{En} R. Engelking, Theory of Dimensions Finite and Infinite, Heldermann Verlag, 1995.

\bibitem[F1]{F1}  Farber, Michael, Topological complexity of motion planning. Discrete Comput. Geom. 29 (2003), no. 2, 211
-221.

\bibitem[F2]{F2}  Farber, Michael, Invitation to topological robotics. Zurich Lectures in Advanced Mathematics. European Mathematical Society (EMS), Zürich, 2008.

\bibitem [F3]{F3} M. Farber, {\em Instability of robot motion.} Top. Appl. 140 (2004), 245-266.

\bibitem[K]{K} A. N. Kolmogorov, Representation of functions of many variables, Dokl. Akad. Nauk 114 (1957), 953-956; English transl., Amer. Math. Soc. Transl. (2) 17 (1961), 369-373.


\bibitem[LM]{LM} W. Lubawski and W. Marzantowicz, A new approach to the equivariant topological complexity,
arXiv: 1303.0171v2 [math.AT].

\bibitem[Os]{Os} Ostrand, Ph. : Dimension of metric spaces and Hilbert's
problem $13$, Bull. Amer. Math. Soc. 71 1965, 619-622.

\bibitem[Pa]{Pa} R. Palais, On the existence of sclices for actions of non-compact Lie groups, Ann. Math., vol 73 No 2, 295-323

\bibitem[Ru]{Ru} Yu. Rudyak, On topological complexity of Eilenberg-MacLane spaces,  arXiv:1302.1238


\bibitem[Sr]{Sr} T. Srinivasan, On the Lusternik-Schnirelmann category of Peano continua,
 Topology Appl. 160 (2013), no. 13, 1742-1749. math arXiv:1212.0899. 

\end{thebibliography}
\end{document}